\documentclass[twoside,10.5pt]{article}
\usepackage{amsfonts}
\usepackage{mathrsfs}
\usepackage{pifont}
\usepackage{amsmath}
\usepackage{amsthm}
\usepackage{txfonts}
\usepackage{geometry}
\usepackage{latexsym}
\usepackage{amssymb}
\usepackage{graphicx}
\usepackage{geometry}
\usepackage{xcolor}

\input{tcilatex}

\begin{document}

\title{\parbox{\linewidth}{\footnotesize\noindent } Prolongation of Poisson $%
2$-form on Weil bundles}
\author{Norbert\ MAHOUNGOU\ MOUKALA\thanks{{\footnotesize nmahomouk@yahoo.fr}%
}, Basile Guy Richard BOSSOTO\thanks{{\footnotesize bossotob@yahoo.fr}} ~}
\date{}
\maketitle

\begin{abstract}
In this paper, $M$\ denotes a smooth manifold of dimension $n$, $A$ a Weil
algebra and $M^{A}$ the associated Weil bundle. When $(M,\omega _{M})$ is a
Poisson manifold with $2$-form $\omega _{M}$, we construct the $2$-Poisson
form $\omega _{M^{A}}^{A}$, prolongation on $M^{A}$ of the $2$-Poisson form $%
\omega _{M}$. We give a necessary and sufficent condition for that $M^{A}$
be an $A$-Poisson manifold.
\end{abstract}



\section{Introduction}

\subsection{Weil algebra and Weil bundle}

In what follows, all structures are assumed to be of class $C^{\infty }$. We
denote by $M$ a smooth differential manifold, $C^{\infty }(M)$ the algebra
of differentiable functions on $M$ and by $\mathfrak{X}(M)$, the $C^{\infty
}(M)$-module of vectors field on $M$.

A Weil algebra is a real, unitary, commutative algebra of finite dimension
with a unique maximal ideal of codimension $1$ on $\mathbb{R}$ \cite{wei}.%
\newline
Let $A$ be a Weil algebra and $\mathfrak{m}$ be its maximal ideal. We have $%
A=\mathbb{R}\oplus \mathfrak{m}$ and the first projection 
\begin{equation*}
A=\mathbb{R}\oplus \mathfrak{m}\longrightarrow \mathbb{R}
\end{equation*}%
is a homomorphism of algebras which is surjective, called augmentation and
the unique non zero integer $h\in \mathbb{N}$ such that $\mathfrak{m}%
^{h}\neq (0)$ and $\mathfrak{m}^{h+1}=(0)$ is the height of $A$ \cite{wei}.%
\newline

If $M$ is a smooth manifold, and $A$ a Weil algebra of maximal ideal $%
\mathfrak{m}$, an infinitely near point to $x\in M$ of kind $A$ is a
homomorphism of algebras 
\begin{equation*}
\xi :C^{\infty }(M)\longrightarrow A
\end{equation*}%
such that $[\xi (f)-f(x)]\in \mathfrak{m}$ for any $f\in C^{\infty }(M)$.
i.e the real part of $\xi (f)$\ is exactly $f(x)$ \cite{wei}.

We denote by $M_{x}^{A}$ the set of all infinitely near points to $x\in M$
of kind $A$ and $M^{A}=\bigcup\limits_{x\in M}M_{x}^{A}$ the manifold of
infinitely near points of kind $A$. We have $\dim M^{A}=\dim M\times \dim A$%
\cite{mor}.\newline

When both $M$ and $N$ are smooth manifolds and when $\ h:M\longrightarrow N$
is a differentiable application, then the map%
\begin{equation}
h^{A}:M^{A}\longrightarrow N^{A},\xi \longmapsto h^{A}(\xi ),  \notag
\end{equation}%
such that, for any $g\in C^{\infty }(N)$, 
\begin{equation}
\left[ h^{A}(\xi )\right] (g)=\xi (g\circ h)  \notag
\end{equation}%
is also differentiable. When $h$ is a diff{e}omorphism, it is the same for $%
h^{A}$ \cite{bok2}.\newline

Moreover, if $\varphi :A\longrightarrow B$ is a homomorphism of Weil
algebras, for any smooth manifold $M$, the map 
\begin{equation}
\varphi _{M}:M^{A}\longrightarrow M^{B},\xi \longmapsto \varphi \circ \xi 
\notag
\end{equation}%
is differentiable. In particular, the augmentation%
\begin{equation}
A\longrightarrow \mathbb{R}  \notag
\end{equation}%
defines for any smooth manifold $M$, the projection 
\begin{equation}
\pi _{M}:M^{A}\longrightarrow M,  \notag
\end{equation}%
which assigns every infinitely near point to $x\in M$ to its origin $x$.
Thus ($M^{A},\pi _{M},M$) defines the bundle of infinitely near points or
simply Weil bundle \cite{kms},\cite{mor},\cite{wei},\cite{oka1}.

If $(U,\varphi )$ is a local chart of $M$ with coordinate functions $%
(x_{1},x_{2},...,x_{n})$, the application 
\begin{equation}
U^{A}\longrightarrow A^{n},\xi \longmapsto (\xi (x_{1}),\xi (x_{2}),...,\xi
(x_{n})),  \notag
\end{equation}%
is a bijection from $U^{A}$ into an open of $A^{n}$. The manifold $M^{A}$ is
a smooth manifold modeled over $A^{n}$, that is to say an $A$-manifold of
dimension $n$ \cite{bok1},\cite{shu}.\newline

The set, $C^{\infty }(M^{A},A)$ of differentiable functions on $M^{A}$ with
values in $A$ is a commutative, unitary algebra over $A$. When one
identitifies $\mathbb{R}^{A}$ with $A$, for $f\in C^{\infty }(M)$, the map%
\begin{equation}
f^{A}:M^{A}\longrightarrow A,\xi \longmapsto \xi (f)  \notag
\end{equation}%
is differentiable and the map%
\begin{equation}
C^{\infty }(M)\longrightarrow C^{\infty }(M^{A},A),f\longmapsto f^{A}\text{,}
\notag
\end{equation}%
is an injective homomorphism of algebras and we have:%
\begin{equation}
(f+g)^{A}=f^{A}+g^{A};(\lambda \cdot f)^{A}=\lambda \cdot f^{A};(f\cdot
g)^{A}=f^{A}\cdot g^{A}  \notag
\end{equation}%
for $\lambda \in \mathbb{R},$ $f,g\in C^{\infty }(M)$.\newline

We denote $\mathfrak{X}(M^{A})$, the set of all vector fields on $M^{A}.$
According to \cite{bok1}, \cite{nbo} We have the following equivalent
assertions:

\begin{theorem}
The following assertions are equivalent:

\begin{enumerate}
\item A vector field on $M^{A}$ is a differentiable section of the tangent
bundle $(TM^{A},\pi _{M^{A}},M^{A})$.

\item A vector field on $M^{A}$ is a derivation of $C^{\infty }(M^{A})$.

\item A vector field on $M^{A}$ is a derivation of $C^{\infty }(M^{A},A)$
which is $A$-linear.

\item A vector field on $M^{A}$ is a linear map $X:C^{\infty
}(M)\longrightarrow C^{\infty }(M^{A},A)$ such that 
\begin{equation*}
X(f\cdot g)=X(f)\cdot g^{A}+f^{A}\cdot X(g),\quad \text{for any}\,f,g\in
C^{\infty }(M)\text{.}
\end{equation*}
\end{enumerate}
\end{theorem}

Consequenly \cite{nbo},

\begin{theorem}
The map 
\begin{equation*}
\mathfrak{X}(M^{A})\times \mathfrak{X}(M^{A})\longrightarrow \mathfrak{X}%
(M^{A}),(X,Y)\longmapsto \lbrack X,Y]=X\circ Y-Y\circ X
\end{equation*}%
is skew-symmetric $A$-bilinear and defines a structure of $A$-Lie algebra
over $\mathfrak{X}(M^{A})$.
\end{theorem}

Thus, if $Der_{A}[C^{\infty }(M^{A},A)]$ denotes the $C^{\infty }(M^{A},A)$%
-module of derivations of $C^{\infty }(M^{A},A)$ which are $A$-linear, a
vector field on $M^{A}$ is a derivation of $C^{\infty }(M^{A},A)$ which is $%
A $-linear i.e a $A$-linear map 
\begin{equation*}
X:C^{\infty }(M^{A},A)\longrightarrow C^{\infty }(M^{A},A)
\end{equation*}%
such that 
\begin{equation*}
X(\varphi \cdot \psi )=X(\varphi )\cdot \psi +\varphi \cdot X(\psi ),\quad 
\text{for any}\,\varphi ,\psi \in C^{\infty }(M^{A},A)\text{.}
\end{equation*}

Thus, we have 
\begin{equation*}
\mathfrak{X}(M^{A})=Der_{A}[C^{\infty }(M^{A},A)]\text{.}
\end{equation*}

\begin{proposition}
\cite{bok1}, \cite{nbo} If $\theta $ $:C^{\infty }(M)\longrightarrow
C^{\infty }(M)$ is a vector field on $M$, then there exists one and only one 
$A$-linear derivation 
\begin{equation*}
\theta ^{A}:C^{\infty }(M^{A},A)\longrightarrow C^{\infty }(M^{A},A)
\end{equation*}%
such that 
\begin{equation*}
\theta ^{A}(f^{A})=\left[ \theta (f)\right] ^{A}
\end{equation*}%
for any $f\in C^{\infty }(M)$.
\end{proposition}

\begin{proposition}
\cite{bok1}, \cite{nbo} If $\theta ,\theta _{1}$ and $\theta _{2}$ are
vector fields on $M$ and if $f\in C^{\infty }(M)$, then we have:

\begin{enumerate}
\item $(\theta _{1}+\theta _{2})^{A}=\theta _{1}^{A}+\theta _{2}^{A}$;

\item $(f\cdot \theta )^{A}=f^{A}\cdot \theta ^{A}$;

\item $[\theta _{1},\theta _{2}]^{A}=[\theta _{1}^{A},\theta _{2}^{A}]$.
\end{enumerate}
\end{proposition}

\begin{corollary}
The map 
\begin{equation*}
\mathfrak{X}(M)\longrightarrow Der_{A}[C^{\infty }(M^{A},A)],\theta
\longmapsto \theta ^{A}
\end{equation*}%
is an injective homomorphism of $\mathbb{R}$-Lie algebras. If $\mu
:A\longrightarrow A$, is a $\mathbb{R}$-endomorphism, and $\theta :C^{\infty
}(M)\longrightarrow C^{\infty }(M)$ a vector field on $M$, then 
\begin{equation*}
\theta ^{A}(\mu \circ f^{A})=\mu \circ \lbrack \theta (f)]^{A},\quad \text{%
for any}\,f\in C^{\infty }(M)\text{.}
\end{equation*}
\end{corollary}

\subsection{\protect\bigskip Poisson manifold}

We recall that a Poisson structure on a smooth manifold $M$ is due to the
existence of a bracket $\{,\}_{M}$ on $C^{\infty }(M)$ such that the pair $%
(C^{\infty }(M),\{,\}_{M})$ is a real Lie algebra such that, for any $f\in
C^{\infty }(M)$ the map

\begin{equation*}
ad(f):C^{\infty }(M)\longrightarrow C^{\infty }(M),g\longmapsto \{f,g\}_{M}
\end{equation*}

is a derivation of commutative algebra i.e

\begin{equation*}
\{f,g\cdot h\}_{M}=\{f,g\}_{M}\cdot h+g\cdot \{f,h\}_{M}
\end{equation*}

for $f,g,h\in $ $C^{\infty }(M)$. In this case we say that $C^{\infty }(M)$
is a Poisson algebra and $M$ is a Poisson manifold \cite{Lic},\cite{vai}, 
\cite{oka2}.

Let $\Omega _{\mathbb{R}}[C^{\infty }(M)]$ be the\ $C^{\infty }(M)$-module
of K\"{a}lher differentials of $C^{\infty }(M)$ and 
\begin{equation*}
\delta _{M}:C^{\infty }(M)\longrightarrow \Omega _{\mathbb{R}}[C^{\infty
}(M)],f\longmapsto \overline{f\otimes 1_{C^{\infty }(M)}-1_{C^{\infty
}(M)}\otimes f}
\end{equation*}%
the canonical derivation which{\LARGE \ }the image of $\delta _{M}$
generates the\ $C^{\infty }(M)$-module $\Omega _{\mathbb{R}}[C^{\infty }(M)]$
i.e for $x\in \Omega _{\mathbb{R}}[C^{\infty }(M)]$,%
\begin{equation*}
x=\sum\limits_{i\in I:finite}f_{i}\cdot \delta _{M}(g_{i})\text{,}
\end{equation*}%
with $f_{i},g_{i}\in C^{\infty }(M)$ for any $i\in I$\cite{cap}, \cite{oka2}%
, \cite{oka3}.

The manifold $M$ is a Poisson manifold if and only if there exists a
skew-symmetric $2$-form 
\begin{equation*}
\omega _{M}:\Omega _{\mathbb{R}}[C^{\infty }(M)]\times \Omega _{\mathbb{R}%
}[C^{\infty }(M)]\longrightarrow C^{\infty }(M)
\end{equation*}%
such that for any $f$ and $g$ in $C^{\infty }(M)$, 
\begin{equation*}
\{f,g\}_{M}=-\omega _{M}[\delta _{M}(f),\delta _{M}(g)]
\end{equation*}%
defines a structure of Lie algebra over $C^{\infty }(M)$ \cite{oka2}, \cite%
{oka3}. In this case, we say that $\omega _{M}$ is the Poisson $2$-form of
the Poisson manifold $M$ and we denote $(M,\omega _{M})$ the Poisson
manifold of Poisson $2$-form $\omega _{M}$.

The main goal of this paper is to study the prolongation of the Poisson $2$%
-form $\omega _{M}$ of Poisson manifold on Weil bundles.

\section{The algebra of K\"{a}hler forms on $C^{\infty }(M^{A},A)$}

\begin{definition}
The\ $C^{\infty }(M^{A})$-module of K\"{a}lher differentials of $C^{\infty
}(M^{A})$ is the set 
\begin{equation*}
\Omega _{\mathbb{R}}[C^{\infty }(M^{A})]=\frac{J}{J^{2}}
\end{equation*}%
where $J\ $is the $C^{\infty }(M^{A})$-submodule of $C^{\infty
}(M^{A})\bigotimes\limits_{\mathbb{R}}C^{\infty }(M^{A})$ generated by the
elements of the form $F\otimes 1_{C^{\infty }(M^{A})}-1_{C^{\infty
}(M^{A})}\otimes F$ with $F\in C^{\infty }(M^{A})$. Thus, the map%
\begin{equation*}
d_{M^{A}}:C^{\infty }(M^{A})\longrightarrow \Omega _{\mathbb{R}}[C^{\infty
}(M^{A})],F\longmapsto \overline{F\otimes 1_{C^{\infty
}(M^{A})}-1_{C^{\infty }(M^{A})}\otimes F}
\end{equation*}%
is a derivation and the image of $d_{M^{A}}$ generates $\Omega _{\mathbb{R}%
}[C^{\infty }(M^{A})]$.
\end{definition}

The $A$-algebra $C^{\infty }(M^{A},A)\bigotimes\limits_{A}C^{\infty
}(M^{A},A)$ admits a structure of $C^{\infty }(M^{A},A)$-module defined by
the homomorphism of $A$-algebras%
\begin{equation*}
C^{\infty }(M^{A},A)\longrightarrow C^{\infty }(M^{A},A)\underset{A}{\otimes 
}C^{\infty }(M^{A},A),\varphi \longmapsto \varphi \otimes 1_{C^{\infty
}(M^{A},A)}\text{.}
\end{equation*}%
In this case, we say that $C^{\infty
}(M^{A},A)\bigotimes\limits_{A}C^{\infty }(M^{A},A)$ admits a structure of $%
C^{\infty }(M^{A},A)$-module defined by the first factor. The second factor
is defined by%
\begin{equation*}
C^{\infty }(M^{A},A)\longrightarrow C^{\infty }(M^{A},A)\underset{A}{\otimes 
}C^{\infty }(M^{A},A),\varphi \longmapsto 1_{C^{\infty }(M^{A},A)}\otimes
\varphi \text{.}
\end{equation*}

The map 
\begin{equation*}
C^{\infty }(M^{A},A)\times C^{\infty }(M^{A},A)\longrightarrow C^{\infty
}(M^{A},A),(\varphi ,\psi )\longmapsto \varphi \cdot \psi
\end{equation*}%
being $A$-bilinear, then there exists a unique $A$-linear map%
\begin{equation*}
m:C^{\infty }(M^{A},A)\underset{A}{\otimes }C^{\infty
}(M^{A},A)\longrightarrow C^{\infty }(M^{A},A)
\end{equation*}%
such that%
\begin{equation*}
m(\varphi \otimes \psi )=\varphi \cdot \psi \text{.}
\end{equation*}%
The kernel$\ $ of $m$ is the $C^{\infty }(M^{A},A)$-submodule of $C^{\infty
}(M^{A},A)\otimes C^{\infty }(M^{A},A)$ generated by the elements of the
form $\varphi \otimes 1_{C^{\infty }(M^{A},A)}-1_{C^{\infty
}(M^{A},A)}\otimes \varphi $ with $\varphi \in C^{\infty }(M^{A},A)$.

We denote $\Omega _{A}[C^{\infty }(M^{A},A)]$, the $C^{\infty }(M^{A},A)$%
-module of K\"{a}lher differentials of $C^{\infty }(M^{A},A)$ which are $A$%
-linears. In this case, for $\varphi \in C^{\infty }(M^{A},A)$, we denote $%
\overline{\varphi \otimes 1_{C^{\infty }(M^{A},A)}-1_{C^{\infty
}(M^{A},A)}\otimes \varphi }$, the class of $\varphi \otimes 1_{C^{\infty
}(M^{A},A)}-1_{C^{\infty }(M^{A},A)}\otimes \varphi $ in $C^{\infty
}(M^{A},A)$.

The map%
\begin{equation*}
C^{\infty }(M)\longrightarrow \Omega _{\mathbb{A}}[C^{\infty
}(M^{A},A)],f\longmapsto \overline{f^{A}\otimes 1_{C^{\infty
}(M^{A},A)}-1_{C^{\infty }(M^{A},A)}\otimes f^{A}}
\end{equation*}%
is a derivation.

Thus,

\begin{proposition}
There exists a unique $A$-linear derivation%
\begin{equation*}
\delta _{M^{A}}^{A}:C^{\infty }(M^{A},A)\longrightarrow \Omega
_{A}[C^{\infty }(M^{A},A)]
\end{equation*}%
such that%
\begin{equation*}
\delta _{M^{A}}^{A}(f^{A})=[\delta _{M}(f)]^{A}
\end{equation*}%
for any $f\in C^{\infty }(M)$.
\end{proposition}

\begin{proof}
Let 
\begin{equation*}
\delta _{M^{A}}^{A}:C^{\infty }(M^{A},A)\overset{\sigma ^{-1}}{%
\longrightarrow }A\otimes C^{\infty }(M^{A})\overset{id_{A}\otimes d_{M^{A}}}%
{\longrightarrow }A\otimes \Omega _{\mathbb{R}}[C^{\infty }(M^{A})]\overset{%
\varpi }{\longrightarrow }\Omega _{A}[C^{\infty }(M^{A},A)]
\end{equation*}%
be that map, where%
\begin{equation*}
\sigma ^{-1}:\varphi =\sum\limits_{\alpha =1}^{\dim A}(a_{\alpha }^{\ast
}\circ \varphi )\cdot a_{\alpha }\longmapsto \sum\limits_{\alpha =1}^{\dim
A}a_{\alpha }\otimes (a_{\alpha }^{\ast }\circ \varphi )
\end{equation*}%
with $(a_{\alpha })_{\alpha =1,\cdot \cdot \cdot ,\dim A}$ a basis of $A$
and $(a_{\alpha }^{\ast })_{\alpha =1,\cdot \cdot \cdot ,\dim A}$ the dual
basis of the basis $(a_{\alpha })_{\alpha =1,\cdot \cdot \cdot ,\dim A}$,%
\begin{equation*}
id_{A}\otimes d_{M^{A}}:\sum\limits_{\alpha =1}^{\dim A}a_{\alpha }\otimes
(a_{\alpha }^{\ast }\circ \varphi )\longmapsto \sum\limits_{\alpha =1}^{\dim
A}a_{\alpha }\otimes d_{M^{A}}(a_{\alpha }^{\ast }\circ \varphi
)=\sum\limits_{\alpha =1}^{\dim A}a_{\alpha }\otimes \left[ \overline{%
(a_{\alpha }^{\ast }\circ \varphi )\otimes 1_{C^{\infty
}(M^{A})}-1_{C^{\infty }(M^{A})}\otimes (a_{\alpha }^{\ast }\circ \varphi })%
\right] \text{,}
\end{equation*}%
\begin{equation*}
\varpi :\sum\limits_{\alpha =1}^{\dim A}a_{\alpha }\otimes
d_{M^{A}}(a_{\alpha }^{\ast }\circ \varphi )\longmapsto \sum\limits_{\alpha
=1}^{\dim A}\left[ \overline{(a_{\alpha }^{\ast }\circ \varphi )a_{\alpha
}\otimes 1_{C^{\infty }(M^{A},A)}-1_{C^{\infty }(M^{A},A)}\otimes (a_{\alpha
}^{\ast }\circ \varphi )a_{\alpha }}\right] \text{.}
\end{equation*}%
Thus, 
\begin{eqnarray*}
\delta _{M^{A}}^{A}(\varphi ) &=&[\varpi \circ (id_{A}\otimes
d_{M^{A}})\circ \sigma ^{-1}](\varphi ) \\
&=&\sum\limits_{\alpha =1}^{\dim A}\overline{\left[ (a_{\alpha }^{\ast
}\circ \varphi )a_{\alpha }\otimes 1_{C^{\infty }(M^{A},A)}-1_{C^{\infty
}(M^{A},A)}\otimes (a_{\alpha }^{\ast }\circ \varphi )a_{\alpha }\right] }%
\text{.}
\end{eqnarray*}%
- For any $\varphi ,\psi \in C^{\infty }(M^{A},A)$, we have%
\begin{eqnarray*}
\delta _{M^{A}}^{A}(\varphi +\psi ) &=&[\varpi \circ (id_{A}\otimes
d_{M^{A}})\circ \sigma ^{-1}](\varphi +\psi ) \\
&=&[\varpi \circ (id_{A}\otimes d_{M^{A}})](\sigma ^{-1}(\varphi )+\sigma
^{-1}(\psi )) \\
&=&[\varpi \circ (id_{A}\otimes d_{M^{A}})](\sigma ^{-1}(\varphi ))+[\varpi
\circ (id_{A}\otimes d_{M^{A}})](\sigma ^{-1}(\psi )) \\
&=&[\varpi \circ (id_{A}\otimes d_{M^{A}})\circ \sigma ^{-1}](\varphi
)+[\varpi \circ (id_{A}\otimes d_{M^{A}})\circ \sigma ^{-1}](\psi ) \\
&=&\delta _{M^{A}}^{A}(\varphi )+\delta _{M^{A}}^{A}(\psi )\text{.}
\end{eqnarray*}%
- For any $\varphi \in C^{\infty }(M^{A},A)$ and $a\in A$, we have%
\begin{eqnarray*}
\delta _{M^{A}}^{A}(a\cdot \varphi ) &=&[\varpi \circ (id_{A}\otimes
d_{M^{A}})\circ \sigma ^{-1}](a\cdot \varphi ) \\
&=&[\varpi \circ (id_{A}\otimes d_{M^{A}})](\sigma ^{-1}(a\cdot \varphi )) \\
&=&[\varpi \circ (id_{A}\otimes d_{M^{A}})](a\cdot \sigma ^{-1}(\varphi )) \\
&=&a\cdot \lbrack \varpi \circ (id_{A}\otimes d_{M^{A}})](\sigma
^{-1}(\varphi )) \\
&=&a\cdot \delta _{M^{A}}^{A}(\varphi )\text{.}
\end{eqnarray*}%
- For any $\varphi ,\psi \in C^{\infty }(M^{A},A)$, we have%
\begin{eqnarray*}
\delta _{M^{A}}^{A}(\varphi \cdot \psi ) &=&[\varpi \circ (id_{A}\otimes
d_{M^{A}})\circ \sigma ^{-1}](\varphi \cdot \psi ) \\
&=&[\varpi \circ (id_{A}\otimes d_{M^{A}})](\sigma ^{-1}(\varphi \cdot \psi
)) \\
&=&[\varpi \circ (id_{A}\otimes d_{M^{A}})](\sigma ^{-1}(\varphi )\cdot
\sigma ^{-1}(\psi )) \\
&=&[\varpi \circ (id_{A}\otimes d_{M^{A}})](\sigma ^{-1}(\varphi ))\cdot
\psi +\varphi \cdot \lbrack \varpi \circ (id_{A}\otimes d_{M^{A}})](\sigma
^{-1}(\psi )) \\
&=&[\varpi \circ (id_{A}\otimes d_{M^{A}})\circ \sigma ^{-1}](\varphi )\cdot
\psi +\varphi \cdot \lbrack \varpi \circ (id_{A}\otimes d_{M^{A}})\circ
\sigma ^{-1}](\psi ) \\
&=&[\varpi \circ (id_{A}\otimes d_{M^{A}})\circ \sigma ^{-1}](\varphi
)+[\varpi \circ (id_{A}\otimes d_{M^{A}})\circ \sigma ^{-1}](\psi ) \\
&=&\delta _{M^{A}}^{A}(\varphi )\cdot \psi +\varphi \cdot \delta
_{M^{A}}^{A}(\psi )\text{.}
\end{eqnarray*}%
As%
\begin{equation*}
\delta _{M}:C^{\infty }(M)\longrightarrow \Omega _{\mathbb{R}}[C^{\infty
}(M)]
\end{equation*}%
is a derivation, then the map 
\begin{equation*}
C^{\infty }(M)\longrightarrow \Omega _{A}[C^{\infty }(M^{A},A)],f\longmapsto
\lbrack \delta _{M}(f)]^{A}
\end{equation*}%
is a derivation. Thus, for any $f\in C^{\infty }(M)$%
\begin{eqnarray*}
\delta _{M^{A}}^{A}\left( f^{A}\right) &=&\varpi \circ \left( id_{A}\otimes
d_{M^{A}}\right) \circ \sigma ^{-1}(f^{A}) \\
&=&\sum\limits_{\alpha =1}^{\dim A}\overline{\left[ (a_{\alpha }^{\ast
}\circ f^{A})a_{\alpha }\otimes 1_{C^{\infty }(M^{A},A)}-1_{C^{\infty
}(M^{A},A)}\otimes (a_{\alpha }^{\ast }\circ f^{A})a_{\alpha }\right] } \\
&=&\overline{\sum\limits_{\alpha =1}^{\dim A}(a_{\alpha }^{\ast }\circ
f^{A})a_{\alpha }\otimes 1_{C^{\infty }(M^{A},A)}-1_{C^{\infty
}(M^{A},A)}\otimes \sum\limits_{\alpha =1}^{\dim A}(a_{\alpha }^{\ast }\circ
f^{A})a_{\alpha }} \\
&=&\overline{f^{A}\otimes 1_{C^{\infty }(M^{A},A)}-1_{C^{\infty
}(M^{A},A)}\otimes f^{A}} \\
&=&\left[ \overline{f\otimes 1_{C^{\infty }(M)}-1_{C^{\infty }(M)}\otimes f}%
\right] ^{A}
\end{eqnarray*}

i.e%
\begin{equation*}
\delta _{M^{A}}^{A}(f^{A})=[\delta _{M}(f)]^{A}\text{.}
\end{equation*}
\end{proof}

\begin{proposition}
The map 
\begin{equation*}
\Omega _{\mathbb{R}}[C^{\infty }(M)]\longrightarrow \Omega _{A}[C^{\infty
}(M^{A},A)],x\longmapsto x^{A}
\end{equation*}%
is an injective homomorphism of $\mathbb{R}$-modules.
\end{proposition}

\begin{proof}
Let%
\begin{equation*}
\Psi :\Omega _{\mathbb{R}}[C^{\infty }(M)]\longrightarrow \Omega
_{A}[C^{\infty }(M^{A},A)],x\longmapsto x^{A}
\end{equation*}%
be that map.
\end{proof}

For any $x,y\in \Omega _{\mathbb{R}}[C^{\infty }(M)]$,%
\begin{eqnarray*}
\Psi (x+y) &=&(x+y)^{A} \\
&=&(\sum\limits_{i\in I:finite}f_{i}\cdot \delta _{M}(f_{i}^{\prime
})+\sum\limits_{j\in I:finite}g_{j}\cdot \delta _{M}(g_{j}^{\prime }))^{A} \\
&=&(\sum\limits_{i\in I:finite}f_{i}\cdot \delta _{M}(f_{i}^{\prime
}))^{A}+(\sum\limits_{j\in I:finite}g_{j}\cdot \delta _{M}(g_{j}^{\prime
}))^{A} \\
&=&x^{A}+y^{A}\text{.}
\end{eqnarray*}%
For any $x\in \Omega _{\mathbb{R}}[C^{\infty }(M)]$ and for $\lambda \in 
\mathbb{R}$,%
\begin{eqnarray*}
\Psi (\lambda \cdot x) &=&(\lambda \cdot x)^{A} \\
&=&(\lambda \cdot \sum\limits_{i\in I:finite}f_{i}\cdot \delta
_{M}(f_{i}^{\prime }))^{A} \\
&=&\lambda \cdot (\sum\limits_{i\in I:finite}f_{i}\cdot \delta
_{M}(f_{i}^{\prime }))^{A} \\
&=&\lambda \cdot x^{A}\text{.}
\end{eqnarray*}%
The pair $(\Omega _{A}[C^{\infty }(M^{A},A)],\delta _{M^{A}}^{A})$ satisfies
the following universal property: for every $C^{\infty }(M^{A},A)$-module $E$
and every $A$-derivation%
\begin{equation*}
\Phi :C^{\infty }(M^{A},A)\longrightarrow E\text{,}
\end{equation*}%
there exists a unique $C^{\infty }(M^{A},A)$-linear map 
\begin{equation*}
\widetilde{\Phi }:\Omega _{A}[C^{\infty }(M^{A},A)]\longrightarrow E
\end{equation*}%
such that 
\begin{equation*}
\widetilde{\Phi }\circ \delta _{M^{A}}^{A}=\Phi \text{.}
\end{equation*}%
In other words, there exists a unique $\widetilde{\Phi }$ which makes the
following diagram commutative%
\begin{equation*}
\begin{array}{ccc}
\Omega _{A}[C^{\infty }(M^{A},A)] &  &  \\ 
\delta _{M^{A}}^{A}\uparrow & \text{ }\overset{\widetilde{\Phi }}{\searrow }
&  \\ 
C^{\infty }(M^{A},A) & \underset{\Phi }{\overset{}{\longrightarrow }} & E%
\text{.}%
\end{array}%
\newline
\end{equation*}%
This fact implies the existence of a natural isomorphism of $C^{\infty
}(M^{A},A)$-modules%
\begin{equation*}
Hom_{C^{\infty }(M^{A},A)}(\Omega _{A}[C^{\infty
}(M^{A},A)],E)\longrightarrow Der_{A}[C^{\infty }(M^{A},A)],E),\psi
\longmapsto \psi \circ \delta _{M^{A}}^{A}\text{.}
\end{equation*}%
In particular, if $E=C^{\infty }(M^{A},A)$, we have 
\begin{eqnarray*}
\Omega _{A}[C^{\infty }(M^{A},A)]^{\ast } &\simeq &Der_{A}[C^{\infty
}(M^{A},A)] \\
&=&\mathfrak{X}(M^{A})\text{.}
\end{eqnarray*}

For any $p\in \mathbb{N}$, $\Lambda ^{p}(\Omega _{A}[C^{\infty }(M^{A},A)])=%
\mathfrak{L}_{sks}^{p}(\Omega _{A}[C^{\infty }(M^{A},A)],C^{\infty
}(M^{A},A))$ denotes the $C^{\infty }(M^{A},A)$-module of skew-symmetric
multilinear forms of degree $p$ from $\Omega _{A}[C^{\infty }(M^{A},A)]$
into $C^{\infty }(M^{A},A)$ and

\begin{equation*}
\Lambda (\Omega _{A}[C^{\infty }(M^{A},A)])=\bigoplus\limits_{p\in \mathbb{N}%
}\Lambda ^{p}(\Omega _{A}[C^{\infty }(M^{A},A)])
\end{equation*}%
the exterior $C^{\infty }(M^{A},A)$-algebra of $\Omega _{A}[C^{\infty
}(M^{A},A)]$.

\begin{equation*}
\Lambda ^{0}(\Omega _{A}[C^{\infty }(M^{A},A)])=C^{\infty }(M^{A},A)\text{,}
\end{equation*}%
\begin{equation*}
\Lambda ^{1}(\Omega _{A}[C^{\infty }(M^{A},A)])=\Omega _{A}[C^{\infty
}(M^{A},A)]^{\ast }\text{.}
\end{equation*}

We denote, 
\begin{equation*}
\delta _{M^{A}}^{A}=\delta _{M^{A}}:\Lambda (\Omega _{A}[C^{\infty
}(M^{A},A)])\longrightarrow \Lambda (\Omega _{A}[C^{\infty }(M^{A},A)])
\end{equation*}%
a unique derivation, of degree $+1$, which extends the canonical derivation 
\begin{equation*}
\delta _{M^{A}}^{0}:C^{\infty }(M^{A},A)\longrightarrow \Omega
_{A}[C^{\infty }(M^{A},A)]\text{.}
\end{equation*}

For any $\varphi ,\psi ,\psi _{1},\psi _{2},...,\psi _{p}\in C^{\infty
}(M^{A},A)$ and $\omega \in \Omega _{A}[C^{\infty }(M^{A},A)]^{\ast }$, we
get

1.%
\begin{equation*}
\delta _{M^{A}}(\varphi \cdot \delta _{M^{A}}(\psi _{1})\wedge ...\wedge
\delta _{M^{A}}(\psi _{p})=\delta _{M^{A}}(\varphi )\wedge \delta
_{M^{A}}(\psi _{1})\wedge ...\wedge \delta _{M^{A}}(\psi _{p})\text{.}
\end{equation*}

2.%
\begin{equation*}
\delta _{M^{A}}^{1}[\psi \cdot \delta _{M^{A}}^{0}(\varphi )]=\delta
_{M^{A}}^{0}(\psi )\wedge \delta _{M^{A}}^{0}(\varphi )\text{.}
\end{equation*}

3.%
\begin{equation*}
\delta _{M^{A}}^{1}(\varphi \cdot \omega )=\delta _{M^{A}}^{0}(\varphi
)\wedge \omega +\varphi \cdot \delta _{M^{A}}^{1}(\omega )\text{.}
\end{equation*}

\begin{proposition}
If $\eta \in \Lambda ^{p}(\Omega _{\mathbb{R}}[C^{\infty }(M)]),$ then $\eta
^{A}\in \Lambda ^{p}(\Omega _{A}[C^{\infty }(M^{A},A)])$.
\end{proposition}

\begin{proof}
Indeed, for any $\eta \in \Lambda ^{p}(\Omega _{\mathbb{R}}[C^{\infty
}(M)]), $ $\eta $\ is of the form $\delta _{M}(f_{1})\wedge ...\wedge \delta
_{M}(f_{p})$ with $f_{1},f_{2},...,f_{p}\in C^{\infty }(M).$%
\begin{eqnarray*}
\eta ^{A} &=&[\delta _{M}(f_{1})\wedge ...\wedge \delta _{M}(f_{p})]^{A} \\
&=&[\delta _{M}(f_{1})]^{A}\wedge ...\wedge \lbrack \delta _{M}(f_{p})]^{A}
\\
&=&\delta _{M^{A}}^{0}(f_{1}^{A})\wedge ...\wedge \delta
_{M^{A}}^{0}(f_{p}^{A})\text{.}
\end{eqnarray*}%
Thus, the $C^{\infty }(M^{A},A)$-module $\Lambda ^{p}(\Omega _{A}[C^{\infty
}(M^{A},A)])$ is generated by elements of the form 
\begin{equation*}
\eta ^{A}=\delta _{M^{A}}^{0}(\varphi _{1})\wedge ...\wedge \delta
_{M^{A}}^{0}(\varphi _{p})
\end{equation*}%
with $\varphi _{1}=f_{1}^{A},...,\varphi _{p}=f_{p}^{A}\in C^{\infty
}(M^{A},A)$.
\end{proof}

The algebra%
\begin{equation*}
\Lambda (\Omega _{A}[C^{\infty }(M^{A},A)])=\bigoplus\limits_{p\in \mathbb{N}%
}\Lambda ^{p}(\Omega _{A}[C^{\infty }(M^{A},A)])
\end{equation*}%
is the algebra of K\"{a}hler forms on $C^{\infty }(M^{A},A)$.

The pair $(\Lambda (\Omega _{A}[C^{\infty }(M^{A},A)],\delta _{M^{A}}^{A}))$
is a differential complex and the map 
\begin{equation*}
A\times \Omega _{\mathbb{R}}[C^{\infty }(M)]\longmapsto \Omega
_{A}[C^{\infty }(M^{A},A)],(a,x)\longmapsto a\cdot x^{A}
\end{equation*}%
induces the morphism of the differential complex $(A\bigotimes \Lambda
(\Omega _{\mathbb{R}}[C^{\infty }(M)]),id_{A}\otimes \delta _{M}))$\ into
the differential

complex $(\Lambda (\Omega _{A}[C^{\infty }(M^{A},A)]),\delta _{M^{A}}^{A})$.

\section{\protect\bigskip Lie derivative with respect to a derivation on $%
M^{A}$}

Let 
\begin{equation*}
\theta :C^{\infty }(M)\longrightarrow C^{\infty }(M)
\end{equation*}%
be a derivation and%
\begin{equation*}
\sigma _{\theta }:\left[ \Omega _{\mathbb{R}}[C^{\infty }(M)]\right]
^{p}\longrightarrow \Lambda ^{p}(\Omega _{\mathbb{R}}[C^{\infty }(M)])\text{,%
}
\end{equation*}%
be the $C^{\infty }(M)$-skew-symmetric multilinear map such that for any $%
x_{1},x_{2},...,x_{p}\in \Omega _{\mathbb{R}}[C^{\infty }(M)]$, 
\begin{equation*}
\sigma _{\theta }(x_{1},x_{2},\cdot \cdot \cdot
,x_{p})=\sum\limits_{i=1}^{p}(-1)^{i-1}\widetilde{\theta }(x_{i})\cdot
x_{1}\wedge ...\wedge \widehat{x_{i}}\wedge ...\wedge x_{p}\text{,}
\end{equation*}
where 
\begin{equation*}
\widetilde{\theta }:\Omega _{\mathbb{R}}[C^{\infty }(M)]\longrightarrow
C^{\infty }(M)
\end{equation*}%
is a unique$C^{\infty }(M)$-linear map such that $\widetilde{\theta }\circ
\delta _{M}=\theta $. Then, 
\begin{equation*}
\sigma _{\theta ^{A}}^{A}:\left[ \Omega _{A}[C^{\infty }(M^{A},A)]\right]
^{p}\longrightarrow \Lambda ^{p}(\Omega _{A}[C^{\infty }(M^{A},A)])
\end{equation*}%
is a unique $C^{\infty }(M^{A},A)$-skew-symmetric multilinear map such that%
\begin{equation*}
\sigma _{\theta ^{A}}^{A}(x_{1}^{A},x_{2}^{A},...,x_{p}^{A})=[\sigma
_{\theta }(x_{1},x_{2},...,x_{p})]^{A}\text{.}
\end{equation*}

We denote%
\begin{equation*}
\widetilde{\sigma _{\theta ^{A}}^{A}}:\Lambda ^{p}(\Omega _{A}[C^{\infty
}(M^{A},A)]\longrightarrow \Lambda ^{p-1}(\Omega _{A}[C^{\infty }(M^{A},A)])%
\text{,}
\end{equation*}%
the unique $C^{\infty }(M^{A},A)$-skew-symmetric multilinear map such that

\begin{equation*}
\widetilde{\sigma _{\theta ^{A}}^{A}}(x_{1}^{A}\wedge x_{2}^{A}\wedge \cdot
\cdot \cdot \wedge x_{p}^{A})=\sigma _{\theta
^{A}}^{A}(x_{1}^{A},x_{2}^{A},\cdot \cdot \cdot ,x_{p}^{A})
\end{equation*}%
i.e. $\sigma _{\theta ^{A}}^{A}$ induces a derivation 
\begin{equation*}
i_{\theta ^{A}}=\widetilde{\sigma _{\theta ^{A}}^{A}}:\Lambda (\Omega
_{A}[C^{\infty }(M^{A},A)])\longrightarrow \Lambda (\Omega _{A}[C^{\infty
}(M^{A},A)])
\end{equation*}%
of degree $-1$.

\begin{proposition}
For any $\theta \in Der_{\mathbb{R}}[C^{\infty }(M)]$ and for any $\eta \in
\Lambda ^{p}(\Omega _{\mathbb{R}}[C^{\infty }(M)])$, we have 
\begin{equation*}
i_{\theta ^{A}}(\eta ^{A})=[i_{\theta }(\eta )]^{A}\text{.}
\end{equation*}
\end{proposition}

\begin{proof}
If $\eta \in \Lambda ^{p}(\Omega _{\mathbb{R}}[C^{\infty }(M)])$, then there
exists $f_{1},f_{2},...,f_{p}\in C^{\infty }(M)$, such that $\eta =\delta
_{M}(f_{1})\wedge ...\wedge \delta _{M}(f_{p})$. Thus, 
\begin{eqnarray*}
i_{\theta ^{A}}(\eta ^{A}) &=&i_{\theta ^{A}}(\left[ \delta
_{M}(f_{1})\wedge ...\wedge \delta _{M}(f_{p})\right] ^{A}) \\
&=&i_{\theta ^{A}}([\delta _{M}(f_{1})]^{A}\wedge ...\wedge \lbrack \delta
_{M}(f_{p})]^{A}) \\
&=&\sigma _{\theta ^{A}}^{A}([\delta _{M}(f_{1})]^{A},...,[\delta
_{M}(f_{p})]^{A}) \\
&=&[\sigma _{\theta }(\delta _{M}(f_{1}),...,\delta _{M}(f_{p}))]^{A} \\
&=&[i_{\theta }(\delta _{M}(f_{1})\wedge ...\wedge \delta _{M}(f_{p})]^{A})
\\
&=&[i_{\theta }(\eta )]^{A}\text{.}
\end{eqnarray*}%
For $p=1$, we have%
\begin{equation*}
i_{\theta ^{A}}=\widetilde{\sigma _{\theta ^{A}}^{A}}:\Lambda ^{1}(\Omega
_{A}[C^{\infty }(M^{A},A)])=\Omega _{A}[C^{\infty }(M^{A},A)]^{\ast
}\longrightarrow \Lambda ^{0}(\Omega _{A}[C^{\infty }(M^{A},A)])=C^{\infty
}(M^{A},A)\text{,}
\end{equation*}%
and for any $y\in \Omega _{\mathbb{R}}[C^{\infty }(M)]$, 
\begin{equation*}
\ i_{\theta ^{A}}(y^{A})=\widetilde{\theta ^{A}}(y^{A})\text{.}
\end{equation*}%
For $p=2$, we have%
\begin{equation*}
\sigma _{\theta ^{A}}^{A}:\Omega _{A}[C^{\infty }(M^{A},A)]\times \Omega
_{A}[C^{\infty }(M^{A},A)]\longrightarrow C^{\infty }(M^{A},A)
\end{equation*}%
and for any $x,y\in \Omega _{\mathbb{R}}[C^{\infty }(M)]$, 
\begin{equation*}
\sigma _{\theta ^{A}}^{A}(x^{A},y^{A})=\widetilde{\theta ^{A}}(x^{A})\cdot
y^{A}-\widetilde{\theta ^{A}}(y^{A})\cdot x^{A}\text{.}
\end{equation*}%
Thus, the map%
\begin{equation*}
i_{\theta ^{A}}:\Lambda ^{2}(\Omega _{A}[C^{\infty
}(M^{A},A)])\longrightarrow \Omega _{A}[C^{\infty }(M^{A},A)]^{\ast }
\end{equation*}%
is the unique $C^{\infty }(M^{A},A)$-linear map such that 
\begin{equation*}
i_{\theta ^{A}}(x^{A}\wedge y^{A})=\sigma _{\theta ^{A}}^{A}(x^{A},y^{A})=%
\widetilde{\theta ^{A}}(x^{A})\cdot y^{A}-\widetilde{\theta ^{A}}%
(y^{A})\cdot x^{A}\text{.}
\end{equation*}
\end{proof}

\begin{definition}
The Lie derivative with respect to $D\in Der_{A}[C^{\infty }(M^{A},A)]$ is
the derivation of degree $0$ 
\begin{equation*}
\mathfrak{L}_{D}=i_{D}\circ \delta _{M^{A}}^{A}+\delta _{M^{A}}^{A}\circ
i_{D}:\Lambda (\Omega _{A}[C^{\infty }(M^{A},A)])\longrightarrow \Lambda
(\Omega _{A}[C^{\infty }(M^{A},A)])\text{.}
\end{equation*}
\end{definition}

\begin{proposition}
For any $\theta \in \mathfrak{X}(M)$, the map%
\begin{equation*}
\mathfrak{L}_{\theta ^{A}}:\Lambda (\Omega _{A}[C^{\infty
}(M^{A},A)])\longrightarrow \Lambda (\Omega _{A}[C^{\infty }(M^{A},A)])
\end{equation*}%
is a unique $A$-linear derivation such that 
\begin{equation*}
\mathfrak{L}_{\theta ^{A}}(\eta ^{A})=[\mathfrak{L}_{\theta }(\eta )]^{A}%
\text{,}
\end{equation*}%
for any $\eta \in \Lambda (\Omega _{\mathbb{R}}[C^{\infty }(M)])$.
\end{proposition}

\begin{proof}
For any $\eta \in \Lambda (\Omega _{\mathbb{R}}[C^{\infty }(M)])$, we have%
\begin{eqnarray*}
\mathfrak{L}_{\theta ^{A}}(\eta ^{A}) &=&i_{\theta ^{A}}[\delta
_{M^{A}}^{A}(\eta ^{A})]+\delta _{M^{A}}^{A}[i_{\theta ^{A}}(\eta ^{A})] \\
&=&i_{\theta ^{A}}([\delta _{M}(\eta )]^{A})+\delta _{M^{A}}^{A}([i_{\theta
}(\eta )]^{A}) \\
&=&(i_{\theta }[\delta _{M}(\eta )])^{A}+(\delta _{M}[i_{\theta }(\eta
)])^{A} \\
&=&(i_{\theta }[\delta _{M}(\eta )]+\delta _{M}[i_{\theta }(\eta )])^{A} \\
&=&[\mathfrak{L}_{\theta }(\eta )]^{A}\text{.}
\end{eqnarray*}
\end{proof}

\begin{proposition}
For any $\theta \in \mathfrak{X}(M)$, for any $x\in \Omega _{\mathbb{R}%
}[C^{\infty }(M)]$ and for any $f\in C^{\infty }(M)$, we have

1.%
\begin{equation*}
\mathfrak{L}_{f^{A}\cdot \theta ^{A}}(x^{A})=[\mathfrak{L}_{f\cdot \theta
}(x)]^{A}\text{.}
\end{equation*}

2.%
\begin{equation*}
\mathfrak{L}_{\theta ^{A}}(f^{A}\cdot x^{A})=[\mathfrak{L}_{\theta }(f\cdot
x)]^{A}\text{.}
\end{equation*}

3.%
\begin{equation*}
\mathfrak{L}_{\theta ^{A}}\left[ \delta _{M^{A}}^{A}(f^{A})\right] =[%
\mathfrak{L}_{\theta }(\delta _{M}(f))]^{A}\text{.}
\end{equation*}
\end{proposition}

\begin{proof}
For any $\theta \in \mathfrak{X}(M)$, for any $x\in \Omega _{\mathbb{R}%
}[C^{\infty }(M)]$ and for any $f\in C^{\infty }(M)$, we have

1. 
\begin{eqnarray*}
\mathfrak{L}_{f^{A}\cdot \theta ^{A}}(x^{A}) &=&i_{f^{A}\cdot \theta
^{A}}[\delta _{M^{A}}^{A}(x^{A})]+\delta _{M^{A}}^{A}[i_{f^{A}\cdot \theta
^{A}}(x^{A})] \\
&=&f^{A}\cdot i_{\theta ^{A}}[\delta _{M^{A}}^{A}(x^{A})]+\delta
_{M^{A}}^{A} \left[ f^{A}\cdot i_{\theta ^{A}}(x^{A})\right] \\
&=&f^{A}\cdot i_{\theta ^{A}}([\delta _{M}(x)]^{A})+\delta
_{M^{A}}^{A}(f^{A}\cdot \lbrack i_{\theta }(x)]^{A}) \\
&=&f^{A}\cdot i_{\theta ^{A}}([\delta _{M}(x)]^{A})+i_{\theta
^{A}}(x^{A})\cdot \delta _{M^{A}}^{A}(f^{A})+f^{A}\cdot \delta _{M^{A}}^{A} 
\left[ i_{\theta ^{A}}(x^{A})\right] \\
&=&f^{A}\cdot (i_{\theta }[\delta _{M}(x)])^{A}+[i_{\theta }(x)]^{A}\cdot
\lbrack \delta _{M}(f)]^{A}+f^{A}\cdot (\delta _{M}\left[ i_{\theta }(x)%
\right] )^{A} \\
&=&(f\cdot i_{\theta }[\delta _{M}(x)]+i_{\theta }(x)\cdot \delta
_{M}(f)+f\cdot \delta _{M}\left[ i_{\theta }(x)\right] )^{A} \\
&=&(f\cdot i_{\theta }[\delta _{M}(x)]+\delta _{M}(f\cdot \lbrack i_{\theta
}(x)])^{A} \\
&=&(i_{f\cdot \theta }[\delta _{M}(x)]+\delta _{M}[i_{f\cdot \theta
}(x)])^{A} \\
&=&[\mathfrak{L}_{f\cdot \theta }(x)]^{A}\text{.}
\end{eqnarray*}%
Thus,%
\begin{equation*}
\mathfrak{L}_{f^{A}\cdot \theta ^{A}}(x^{A})=[\mathfrak{L}_{f\cdot \theta
}(x)]^{A}\text{.}
\end{equation*}

2.%
\begin{eqnarray*}
\mathfrak{L}_{\theta ^{A}}(f^{A}\cdot x^{A}) &=&i_{\theta ^{A}}\left[ \delta
_{M^{A}}^{A}(f^{A}\cdot x^{A})\right] +\delta _{M^{A}}^{A}\left[ i_{\theta
^{A}}(f^{A}\cdot x^{A})\right] \\
&=&i_{\theta ^{A}}\left[ \delta _{M^{A}}^{A}(f^{A})\Lambda x^{A}+f^{A}\cdot
\delta _{M^{A}}^{A}(x^{A})\right] +\delta _{M^{A}}^{A}\left[ f^{A}\cdot
i_{\theta ^{A}}(x^{A})\right] \\
&=&\theta ^{A}(f^{A})\cdot x^{A}-\delta _{M^{A}}^{A}(f^{A})\cdot \widetilde{%
\theta ^{A}}(x^{A})+f^{A}\cdot \theta ^{A}(x^{A}) \\
&&+\delta _{M^{A}}^{A}(f^{A})\cdot \widetilde{\theta ^{A}}(x^{A})+f^{A}\cdot
\delta _{M^{A}}^{A}\left[ i_{\theta ^{A}}(x^{A})\right] \\
&=&(\theta (f)\cdot x-\delta _{M}(f)\cdot \widetilde{\theta }(x)+f\cdot
\theta (x) \\
&&+\delta _{M}(f)\cdot \widetilde{\theta }(x)+f\cdot \delta _{M}\left[
i_{\theta }(x)\right] )^{A} \\
&=&(i_{\theta }\left[ \delta _{M}(f)\Lambda x+f\cdot \delta _{M}(x)\right]
+\delta _{M}\left[ f\cdot i_{\theta }(x)\right] )^{A} \\
&=&(i_{\theta }\left[ \delta _{M}(f\cdot x)\right] +\delta _{M}\left[
i_{\theta }(f\cdot x)\right] )^{A} \\
&=&[\mathfrak{L}_{\theta }(f\cdot x)]^{A}\text{.}
\end{eqnarray*}

3.%
\begin{eqnarray*}
\mathfrak{L}_{\theta ^{A}}\left[ \delta _{M^{A}}^{A}(f^{A})\right]
&=&i_{\theta ^{A}}\left[ \delta _{M^{A}}^{A}(\delta _{M^{A}}^{A}(f^{A}))%
\right] +\delta _{M^{A}}^{A}\left[ i_{\theta ^{A}}(\delta
_{M^{A}}^{A}(f^{A}))\right] \\
&=&0+\delta _{M^{A}}^{A}\left[ \theta ^{A}(f^{A})\right] \\
&=&\delta _{M^{A}}^{A}(\left[ \theta (f)\right] ^{A}) \\
&=&(\delta _{M}\left[ \theta (f)\right] )^{A} \\
&=&(0+\delta _{M}\left[ \theta (f)\right] )^{A} \\
&=&(i_{\theta }\left[ \delta _{M}(\delta _{M}(f))\right] +\delta _{M}\left[
i_{\theta }(\delta _{M}(f))\right] )^{A} \\
&=&[\mathfrak{L}_{\theta }(\delta _{M}(f))]^{A}\text{.}
\end{eqnarray*}
\end{proof}

\begin{proposition}
For any $D\in Der_{A}[C^{\infty }(M^{A},A)]$, $X\in \Omega _{A}[C^{\infty
}(M^{A},A)]$, and $\varphi \in C^{\infty }(M^{A},A)$, we have

1.%
\begin{equation*}
\mathfrak{L}_{\varphi \cdot D}(X)=\varphi \cdot \mathfrak{L}_{D}(X)+%
\widetilde{D}(X)\cdot \delta _{M^{A}}^{A}(\varphi )\text{;}
\end{equation*}

2.%
\begin{equation*}
\mathfrak{L}_{D}(\varphi \cdot X)=D(\varphi )\cdot X+\varphi \cdot \mathfrak{%
L}_{D}(X)\text{;}
\end{equation*}

3.%
\begin{equation*}
\mathfrak{L}_{D}\left[ \delta _{M^{A}}^{A}(\varphi )\right] =\delta
_{M^{A}}^{A}\left[ D(\varphi )\right] \text{.}
\end{equation*}
\end{proposition}

\begin{proof}
For any $D\in Der_{A}[C^{\infty }(M^{A},A)]$, $X\in \Omega _{A}[C^{\infty
}(M^{A},A)]$, and $\varphi \in C^{\infty }(M^{A},A)$, we have

1. 
\begin{eqnarray*}
\mathfrak{L}_{\varphi \cdot D}(X) &=&i_{\varphi \cdot D}[\delta
_{M^{A}}^{A}(X)]+\delta _{M^{A}}^{A}[i_{\varphi \cdot D}(X)] \\
&=&\varphi \cdot i_{D}\left[ \delta _{M^{A}}^{A}(X)\right] +\delta
_{M^{A}}^{A}\left[ \varphi \cdot i_{D}(X)\right] \\
&=&\varphi \cdot i_{D}\left[ \delta _{M^{A}}^{A}(X)\right] +i_{D}(X)\cdot
\delta _{M^{A}}^{A}(\varphi )+\varphi \cdot \delta _{M^{A}}^{A}\left[
i_{D}(X)\right] \\
&=&\varphi \cdot \mathfrak{L}_{D}(X)+\widetilde{D}(X)\cdot \delta
_{M^{A}}^{A}(\varphi )\text{.}
\end{eqnarray*}

2.%
\begin{eqnarray*}
\mathfrak{L}_{D}(\varphi \cdot X) &=&i_{D}\left[ \delta _{M^{A}}^{A}(\varphi
\cdot X)\right] +\delta _{M^{A}}^{A}\left[ i_{D}(\varphi \cdot X)\right] \\
&=&i_{D}\left[ \delta _{M^{A}}^{A}(\varphi )\Lambda X+\varphi \cdot \delta
_{M^{A}}^{A}(X)\right] +\delta _{M^{A}}^{A}\left[ \varphi \cdot i_{D}(X)%
\right] \\
&=&\widetilde{D}[\delta _{M^{A}}^{A}(\varphi )]\cdot X-\delta
_{M^{A}}^{A}(\varphi )\cdot \widetilde{D}(X)+\varphi \cdot i_{D}[\delta
_{M^{A}}^{A}(X)] \\
&&+\delta _{M^{A}}^{A}(\varphi )\cdot \widetilde{D}(X)+\varphi \cdot \delta
_{M^{A}}^{A}\left[ i_{D}(X)\right] \\
&=&D(\varphi )\cdot X+\varphi \cdot i_{D}[\delta _{M^{A}}^{A}(X)]+\varphi
\cdot \delta _{M^{A}}^{A}\left[ i_{D}(X)\right] \\
&=&D(\varphi )\cdot X+\varphi \cdot (i_{D}[\delta _{M^{A}}^{A}(X)]+\delta
_{M^{A}}^{A}\left[ i_{D}(X)\right] ) \\
&=&D(\varphi )\cdot X+\varphi \cdot \mathfrak{L}_{D}(X)\text{.}
\end{eqnarray*}

3.%
\begin{eqnarray*}
&&\mathfrak{L}_{D}\left[ \delta _{M^{A}}^{A}(\varphi )\right] \\
&=&i_{D}\left[ \delta _{M^{A}}^{A}(\delta _{M^{A}}^{A}(\varphi ))\right]
+\delta _{M^{A}}^{A}\left[ i_{D}(\delta _{M^{A}}^{A}(\varphi ))\right] \\
&=&0+\delta _{M^{A}}^{A}\left[ \widetilde{D}\circ \delta
_{M^{A}}^{A}(\varphi )\right] \\
&=&\delta _{M^{A}}^{A}\left[ D(\varphi )\right] \text{.}
\end{eqnarray*}
\end{proof}

\section{The Poisson $2$-form on Weil bundles}

We recall that, when $M$ is a smooth manifold, $A$ a Weil algebra and $M^{A}$
the associated Weil bundle, the $A$-algebra $C^{\infty }(M^{A},A)$ is a
Poisson algebra over $A$ if there exists a bracket $\left\{ ,\right\} $ on $%
C^{\infty }(M^{A},A)$ such that the pair $\left( C^{\infty
}(M^{A},A),\left\{ ,\right\} \right) $ is a Lie algebra over $A$ satisfying 
\begin{equation*}
\left\{ \varphi ,\psi _{1}\cdot \psi _{2}\right\} =\left\{ \varphi ,\psi
_{1}\right\} \cdot \psi _{2}+\psi _{1}\cdot \left\{ \varphi ,\psi
_{2}\right\}
\end{equation*}%
for any $\varphi ,\psi _{1},\psi _{2}\in C^{\infty }(M^{A},A)$. When $%
C^{\infty }(M^{A},A)$ is a Poisson A-algebra, we will say that the manifold $%
M^{A}$ is a $A$-Poisson manifold \cite{bok2},\cite{mab}.

When $\left( M,\left\{ ,\right\} \right) $ is a Poisson manifold, the map%
\begin{equation*}
ad:C^{\infty }(M)\longrightarrow Der_{\mathbb{R}}[C^{\infty
}(M)],f\longmapsto ad(f)
\end{equation*}%
such that $\left[ ad(f)\right] (g)=\left\{ f,g\right\} $ for any $g\in
C^{\infty }(M)$, is a derivation. Thus :

\begin{proposition}
There exists a derivation%
\begin{equation*}
ad^{A}:C^{\infty }(M^{A},A)\longrightarrow Der_{A}[C^{\infty }(M^{A},A)]
\end{equation*}%
such that%
\begin{equation*}
ad^{A}(f^{A})=[ad(f)]^{A}\text{.}
\end{equation*}
\end{proposition}

Let's consider{\LARGE \ }the following diagram commutative: 
\begin{equation*}
\begin{array}{ccc}
C^{\infty }(M^{A},A) & \overset{\widetilde{\tau }}{\longrightarrow } & 
Der_{A}[C^{\infty }(M^{A},A)] \\ 
\uparrow \gamma _{M} &  & \uparrow \Phi \\ 
C^{\infty }(M) & \overset{ad}{\longrightarrow } & Der_{\mathbb{R}}[C^{\infty
}(M)]%
\end{array}%
\end{equation*}%
i.e. 
\begin{equation*}
\widetilde{\tau }\circ \gamma _{M}=\Phi \circ ad\text{,}
\end{equation*}%
\ where 
\begin{equation*}
\gamma _{M}:C^{\infty }(M)\longrightarrow C^{\infty }(M^{A},A),f\longmapsto
f^{A}
\end{equation*}%
and 
\begin{equation*}
\Phi :Der_{\mathbb{R}}[C^{\infty }(M)]\longrightarrow Der_{A}[C^{\infty
}(M^{A},A)],\theta \longmapsto \theta ^{A}\text{.}
\end{equation*}%
For any $f\in C^{\infty }(M)$, we have%
\begin{equation*}
\widetilde{\tau }\circ \gamma _{M}(f)=\widetilde{\tau }(f^{A})
\end{equation*}%
and%
\begin{eqnarray*}
\Phi \circ ad(f) &=&\Phi \lbrack ad(f)] \\
&=&[ad(f)]^{A}\text{.}
\end{eqnarray*}%
Thus, there exists $ad^{A}=\widetilde{\tau }$ such that%
\begin{equation*}
ad^{A}(f^{A})=[ad(f)]^{A}\text{.}
\end{equation*}%
As 
\begin{equation*}
ad^{A}:C^{\infty }(M^{A},A)\longrightarrow Der_{A}[C^{\infty }(M^{A},A)]
\end{equation*}%
is a derivation, then there exists a unique $C^{\infty }(M^{A},A)$-linear map%
\begin{equation*}
\widetilde{ad^{A}}:\Omega _{A}[C^{\infty }(M^{A},A)]\longrightarrow
Der_{A}[C^{\infty }(M^{A},A)]
\end{equation*}%
such that 
\begin{equation*}
\widetilde{ad^{A}}\circ \delta _{M^{A}}^{A}=ad^{A}\text{.}
\end{equation*}%
Let's consider the canonical isomorphism 
\begin{equation*}
\sigma _{M^{A}}:\Omega _{A}[C^{\infty }(M^{A},A)]^{\ast }\longrightarrow
Der_{A}[C^{\infty }(M^{A},A)],\Psi \longmapsto \Psi \circ \delta _{M^{A}}^{A}
\end{equation*}%
and let%
\begin{equation*}
\sigma _{M^{A}}^{-1}\circ \widetilde{ad^{A}}:\Omega _{A}[C^{\infty
}(M^{A},A)]\overset{\widetilde{ad^{A}}}{\longrightarrow }Der_{A}[C^{\infty
}(M^{A},A)]\overset{\sigma _{M^{A}}^{-1}}{\longrightarrow }\Omega
_{A}[C^{\infty }(M^{A},A)]^{\ast }
\end{equation*}%
be the map.

\begin{proposition}
If $(M,\omega _{M})$ is a Poisson manifold, then the map, 
\begin{equation*}
\omega _{M^{A}}^{A}:\Omega _{A}[C^{\infty }(M^{A},A)]\times \Omega
_{A}[C^{\infty }(M^{A},A)]\longrightarrow C^{\infty }(M^{A},A)
\end{equation*}%
such that for any $X,Y\in \Omega _{A}[C^{\infty }(M^{A},A)]$%
\begin{equation*}
\omega _{M^{A}}^{A}(X,Y)=-[\sigma _{M^{A}}^{-1}\circ \widetilde{ad^{A}}%
(X)](Y)
\end{equation*}%
is a skew-symmetric $2$-form on $\Omega _{A}[C^{\infty }(M^{A},A)]$ such that%
\begin{equation*}
\omega _{M^{A}}^{A}(x^{A},y^{A})=[\omega _{M}(x,y)]^{A}\text{,}
\end{equation*}%
for any $x$ and $y$ in $\Omega _{\mathbb{R}}[C^{\infty }(M)]$.
\end{proposition}

\begin{proof}
For any $X\in \Omega _{A}[C^{\infty }(M^{A},A)]$, we have $%
X=\sum\limits_{i\in I:fini}\varphi _{i}\cdot \delta _{M^{A}}^{A}(\psi _{i}),$
with $\varphi _{i}\in C^{\infty }(M^{A},A)$, $\psi _{i}\in C^{\infty
}(M^{A},A)$.%
\begin{eqnarray*}
\omega _{M^{A}}(X,X) &=&-[\sigma _{M^{A}}^{-1}\circ \widetilde{ad^{A}}(X)](X)
\\
&=&-\sum\limits_{j\in I:finite}\varphi _{j}\cdot \lbrack \sigma
_{M^{A}}^{-1}\circ \widetilde{ad^{A}}(X)]\delta _{M^{A}}^{A}(\psi _{j}) \\
&=&-\sum\limits_{j\in I:finite}\varphi _{j}\cdot \lbrack \widetilde{ad^{A}}%
(X)](\psi _{j}) \\
&=&-\sum\limits_{j,k\in I:finite}\varphi _{j}\cdot \varphi _{k}\cdot \lbrack 
\widetilde{ad^{A}}(\delta _{M^{A}}^{A}(\psi _{k}))](\psi _{j}) \\
&=&-\sum\limits_{j,k\in I:finite}\varphi _{j}\cdot \varphi _{k}\cdot \lbrack
ad^{A}(\psi _{k})](\psi _{j}) \\
&=&-\sum\limits_{j,k\in I:finite}\varphi _{j}\cdot \varphi _{k}\cdot \{\psi
_{k},\psi _{j}\} \\
&=&0\text{.}
\end{eqnarray*}

For any $X_{1},X_{2}$ and $Y\in \Omega _{A}[C^{\infty }(M^{A},A)]$ and for
any $\varphi \in C^{\infty }(M^{A},A)$, we have%
\begin{eqnarray*}
\omega _{M^{A}}[(\varphi \cdot X_{1}+X_{2}),Y] &=&-[\sigma
_{M^{A}}^{-1}\circ \widetilde{ad^{A}}(\varphi \cdot X_{1}+X_{2})](Y) \\
&=&-(\sigma _{M^{A}}^{-1}[\varphi \cdot \widetilde{ad^{A}}(X_{1})+\widetilde{%
ad^{A}}(X_{2})](Y) \\
&=&-\varphi \cdot (\sigma _{M^{A}}^{-1}[\widetilde{ad^{A}}%
(X_{1})](Y)+(\sigma _{M^{A}}^{-1}[\widetilde{ad^{A}}(X_{2})](Y) \\
&=&\varphi \cdot \omega _{M^{A}}(X_{1},Y)+\omega _{M^{A}}(X_{2},Y)\text{.}
\end{eqnarray*}

For any $x$ and $y$ in $\Omega _{\mathbb{R}}[C^{\infty }(M)]$, 
\begin{equation*}
x^{A}=\sum\limits_{i\in I:fini}f_{i}^{A}\cdot \delta
_{M^{A}}^{A}(f_{i}^{\prime A})\text{ and }y^{A}=\sum\limits_{j\in
I:fini}g_{j}^{A}\cdot \delta _{M^{A}}^{A}(g_{j}^{\prime A})\text{.}
\end{equation*}%
Thus,%
\begin{eqnarray*}
\omega _{M^{A}}^{A}(x^{A},y^{A}) &=&-[\sigma _{M^{A}}^{-1}\circ \widetilde{%
ad^{A}}(x^{A})](y^{A}) \\
&=&-[\sigma _{M^{A}}^{-1}\circ \widetilde{ad^{A}}(\sum\limits_{i\in
I:finite}f_{i}^{A}\cdot \delta _{M^{A}}^{A}(f_{i}^{\prime
A}))](\sum\limits_{j\in I:finite}g_{j}^{A}\cdot \delta
_{M^{A}}^{A}(g_{j}^{\prime A})) \\
&=&-\sum\limits_{i,j\in I:finite}f_{i}^{A}\cdot g_{j}^{A}\cdot \lbrack
\sigma _{M^{A}}^{-1}\circ \widetilde{ad^{A}}(\delta
_{M^{A}}^{A}(f_{i}^{\prime A}))](\delta _{M^{A}}^{A}(g_{j}^{\prime A})) \\
&=&-\sum\limits_{i,j\in I:finite}f_{i}^{A}\cdot g_{j}^{A}\cdot \lbrack
\sigma _{M^{A}}^{-1}\circ \widetilde{ad^{A}}(\delta
_{M^{A}}^{A}(f_{i}^{\prime A}))](\delta _{M^{A}}^{A}(g_{j}^{\prime A})) \\
&=&-\sum\limits_{i,j\in I:finite}f_{i}^{A}\cdot g_{j}^{A}\cdot \lbrack
\sigma _{M^{A}}^{-1}(\widetilde{ad^{A}}\circ \delta
_{M^{A}}^{A}(f_{i}^{\prime A}))](\delta _{M^{A}}^{A}(g_{j}^{\prime A})) \\
&=&-\sum\limits_{i,j\in I:finite}f_{i}^{A}\cdot g_{j}^{A}\cdot \lbrack
\sigma _{M^{A}}^{-1}(ad^{A}(f_{i}^{\prime A}))](\delta
_{M^{A}}^{A}(g_{j}^{\prime A})) \\
&=&-\sum\limits_{i,j\in I:finite}f_{i}^{A}\cdot g_{j}^{A}\cdot \lbrack
\sigma _{M^{A}}^{-1}(ad(f_{i}^{\prime })^{A})](\delta _{M}(g_{j}^{\prime
}))^{A} \\
&=&-[\sum\limits_{i,j\in I:finite}f_{i}\cdot g_{j}\cdot \lbrack \sigma
_{M}^{-1}(ad(f_{i}^{\prime }))](\delta _{M}(g_{j}^{\prime }))]^{A} \\
&=&-[\sum\limits_{i,j\in I:finite}f_{i}\cdot g_{j}\cdot \lbrack \sigma
_{M}^{-1}(\widetilde{ad}\circ \delta _{M}(f_{i}^{\prime }))](\delta
_{M}(g_{j}^{\prime }))]^{A} \\
&=&-[\sigma _{M}^{-1}\circ \widetilde{ad}(\sum\limits_{i\in
I:finite}f_{i}\cdot \delta _{M}(f_{i}^{\prime }))](\sum\limits_{j\in
I:finite}g_{j}\cdot \delta _{M}(g_{j}^{\prime }))]^{A} \\
&=&[\omega _{M}(x,y)]^{A}\text{.}
\end{eqnarray*}
\end{proof}

\begin{proposition}
When $(M,\omega _{M})$ is a Poisson manifold of Poisson $2$-form $\omega
_{M},$ then $(M^{A},\omega _{M^{A}}^{A})$ is a Poisson manifold.
\end{proposition}

\begin{proof}
For any $f$ and $g$ in $C^{\infty }(M)$, 
\begin{eqnarray*}
\omega _{M^{A}}^{A}(\delta _{M^{A}}^{A}(f^{A}),\delta _{M^{A}}^{A}(g^{A}))
&=&\omega _{M^{A}}^{A}([\delta _{M}(f)]^{A},[\delta _{M}(g)]^{A}) \\
&=&[\omega _{M}(\delta _{M^{A}}(f),\delta _{M}(g))]^{A} \\
&=&-\{f,g\}_{M}^{A}\text{.}
\end{eqnarray*}%
and 
\begin{equation*}
\omega _{M^{A}}^{A}(x^{A},y^{A})=[\omega _{M}(x,y)]^{A}\text{,}
\end{equation*}%
for any $x,y\in \Omega _{\mathbb{R}}[C^{\infty }(M)]$. We deduce that $%
(M^{A},\omega _{M^{A}}^{A})$ is a Poisson manifold.
\end{proof}

\begin{theorem}
The manifold $M^{A}$ is a Poisson manifold if and only if there exists a
skew-symmetric $2$-form 
\begin{equation*}
\omega _{M^{A}}^{A}:\Omega _{A}[C^{\infty }(M^{A},A)]\times \Omega
_{A}[C^{\infty }(M^{A},A)]\longrightarrow C^{\infty }(M^{A},A)
\end{equation*}%
such that for any $\varphi $ and $\psi $ in $C^{\infty }(M^{A},A)$, 
\begin{equation*}
\{\varphi ,\psi \}_{M^{A}}=-\omega _{M^{A}}^{A}(\delta _{M^{A}}^{A}(\varphi
),\delta _{M^{A}}^{A}(\psi ))
\end{equation*}%
defines a structure of $A$-Lie algebra over $C^{\infty }(M^{A},A)$.
Moreover, for any $f$ and $g$ in $C^{\infty }(M)$,%
\begin{equation*}
\{f^{A},g^{A}\}_{M^{A}}=\{f,g\}_{M}^{A}\text{.}
\end{equation*}
\end{theorem}

\begin{proof}
Indeed{\LARGE , }according to the previous proposition, the{\LARGE \ }%
bracket 
\begin{equation*}
\{\varphi ,\psi \}_{M^{A}}=-\omega _{M^{A}}^{A}(\delta _{M^{A}}^{A}(\varphi
),\delta _{M^{A}}^{A}(\psi ))
\end{equation*}%
defines a structure of $A$-Lie algebra over $C^{\infty }(M^{A},A)$. For any $%
f$ and $g$ in $C^{\infty }(M)$,%
\begin{eqnarray*}
\{f^{A},g^{A}\}_{M^{A}} &=&-\omega _{M^{A}}^{A}(\delta
_{M^{A}}^{A}(f^{A}),\delta _{M^{A}}^{A}(g^{A})) \\
&=&\{f,g\}_{M}^{A}\text{.}
\end{eqnarray*}%
In this case, we will say that $\omega _{M^{A}}^{A}$ is the Poisson $2$-form
of the $A$-Poisson manifold $M^{A}$ and we denote $(M^{A},\omega
_{M^{A}}^{A})$ the $A$-Poisson manifold of Poisson $2$-form $\omega
_{M^{A}}^{A}$.
\end{proof}

\begin{proposition}
When $(M,\omega _{M})$ is a Poisson manifold of Poisson $2$-form $\omega
_{M} $, then for any $x,y\in \Omega _{\mathbb{R}}[C^{\infty }(M)]$ and for
any $f,g\in $ $C^{\infty }(M)$, we get

$1.$ 
\begin{equation*}
\lbrack \widetilde{ad^{A}}(x^{A})](f^{A})=([\widetilde{ad}(x)](f))^{A}\text{.%
}
\end{equation*}

$2.$%
\begin{equation*}
\lbrack \widetilde{\widetilde{ad^{A}}(x^{A})}](y^{A})=([\widetilde{%
\widetilde{ad}(x)}](y))^{A}\text{.}
\end{equation*}

$3.$ 
\begin{equation*}
\mathfrak{L}_{\widetilde{ad^{A}}[\delta _{M^{A}}^{A}(f^{A})]}(g^{A})=(%
\mathfrak{L}_{\widetilde{ad}[\delta _{M}(f)]}(g))^{A}\text{.}
\end{equation*}
\end{proposition}

\begin{proof}
$1.$ For any $x\in \Omega _{\mathbb{R}}[C^{\infty }(M)],x^{A}=g^{A}\cdot
\delta _{M^{A}}^{A}(h^{A})$ with $g$ and $h$ in $C^{\infty }(M)$, and for
any $f,g\in $ $C^{\infty }(M)$, we have 
\begin{eqnarray*}
\lbrack \widetilde{ad^{A}}(x^{A})](f^{A}) &=&[\widetilde{ad^{A}}(g^{A}\cdot
\delta _{M^{A}}^{A}(h^{A}))](f^{A}) \\
&=&[g^{A}\cdot \widetilde{ad^{A}}\circ \delta _{M^{A}}^{A}(h^{A}))](f^{A}) \\
&=&[g^{A}\cdot ad^{A}(h^{A})](f^{A}) \\
&=&g^{A}\cdot \lbrack ad(h)]^{A}(f^{A}) \\
&=&(g\cdot \lbrack ad(h)](f))^{A} \\
&=&(g\cdot \widetilde{ad}[\delta _{M}(h)](f))^{A} \\
&=&(\widetilde{ad}[g\cdot \delta _{M}(h)](f))^{A} \\
&=&(\widetilde{ad}[g\cdot \delta _{M}(h)](f))^{A} \\
&=&([\widetilde{ad}(x)](f))^{A}\text{.}
\end{eqnarray*}

$2.$ When $y\in \Omega _{\mathbb{R}}[C^{\infty }(M)],y^{A}=g^{A}\cdot \delta
_{M^{A}}^{A}(h^{A})$ with $g$ and $h$ in $C^{\infty }(M)$%
\begin{eqnarray*}
\lbrack \widetilde{\widetilde{ad^{A}}(x^{A})}](y^{A}) &=&[\widetilde{%
\widetilde{ad^{A}}(x^{A})}](g^{A}\cdot \delta _{M^{A}}^{A}(h^{A})) \\
&=&g^{A}\cdot ([\widetilde{\widetilde{ad^{A}}(x^{A})}]\circ \delta
_{M^{A}}^{A})(h^{A}) \\
&=&g^{A}\cdot \lbrack \widetilde{ad^{A}}(x^{A})](h^{A}) \\
&=&g^{A}\cdot \lbrack \widetilde{ad^{A}}(x^{A})](h^{A}) \\
&=&(g\cdot \lbrack \widetilde{ad}(x)](h))^{A} \\
&=&(g\cdot \lbrack \widetilde{\widetilde{ad}(x)}]\circ \delta _{M})(h))^{A}
\\
&=&([\widetilde{\widetilde{ad}(x)}](g\cdot \delta _{M}(h)))^{A} \\
&=&([\widetilde{\widetilde{ad}(x)}](y))^{A}\text{.}
\end{eqnarray*}
\end{proof}

\begin{proposition}
When $(M,\omega _{M})$ is a Poisson manifold of Poisson $2$-form $\omega
_{M} $, then for any $X,Y\in \Omega _{A}[C^{\infty }(M^{A},A)]$ and for any $%
\varphi ,\psi \in C^{\infty }(M^{A},A)$, we get

1. 
\begin{equation*}
\lbrack \widetilde{ad^{A}}(X)](\varphi )=-\omega _{M^{A}}^{A}(X,\delta
_{M^{A}}^{A}(\varphi ))\text{;}
\end{equation*}

2.%
\begin{equation*}
\lbrack \widetilde{\widetilde{ad^{A}}(X)}](Y)=-\omega _{M^{A}}^{A}(X,Y)\text{%
;}
\end{equation*}

3. 
\begin{equation*}
\mathfrak{L}_{\widetilde{ad^{A}}[\delta _{M^{A}}^{A}(\varphi )]}\delta
_{M^{A}}^{A}(\psi )=\delta _{M^{A}}^{A}(\{\varphi ,\psi \}_{M^{A}})\text{.}
\end{equation*}
\end{proposition}

\begin{proof}
When $X$ and $Y\in \Omega _{A}[C^{\infty }(M^{A},A)],$ $X=\sum\limits_{i\in
I:fini}\varphi _{i}\cdot \delta _{M^{A}}^{A}(\varphi _{i}^{\prime }),$ $%
Y=\sum\limits_{j\in J:fini}\psi _{j}\cdot \delta _{M^{A}}^{A}(\psi
_{i}^{\prime })$ with $\varphi _{i},\varphi _{i}^{\prime },\psi _{j},\psi
_{j}^{\prime }\in C^{\infty }(M^{A},A)$

1. 
\begin{eqnarray*}
\lbrack \widetilde{ad^{A}}(X)](\varphi ) &=&[\widetilde{ad^{A}}%
(\sum\limits_{i\in I:fini}\varphi _{i}\cdot \delta _{M^{A}}^{A}(\varphi
_{i}^{\prime }))](\varphi ) \\
&=&\sum\limits_{i\in I:fini}\varphi _{i}\cdot (\widetilde{ad^{A}}[\delta
_{M^{A}}^{A}(\varphi _{i}^{\prime })])(\varphi ) \\
&=&\sum\limits_{i\in I:fini}\varphi _{i}\cdot \lbrack ad^{A}(\varphi
_{i}^{\prime })](\varphi ) \\
&=&\sum\limits_{i\in I:fini}\varphi _{i}\cdot \{\varphi _{i}^{\prime
},\varphi \}_{M^{A}} \\
&=&-\sum\limits_{i\in I:fini}\varphi _{i}\cdot \omega _{M^{A}}^{A}(\delta
_{M^{A}}^{A}(\varphi _{i}^{\prime }),\delta _{M^{A}}^{A}(\varphi )) \\
&=&-\omega _{M^{A}}^{A}(\sum\limits_{i\in I:fini}\varphi _{i}\cdot \delta
_{M^{A}}^{A}(\varphi _{i}^{\prime }),\delta _{M^{A}}^{A}(\varphi )) \\
&=&-\omega _{M^{A}}^{A}(X,\delta _{M^{A}}^{A}(\varphi ))\text{.}
\end{eqnarray*}

2. 
\begin{eqnarray*}
\lbrack \widetilde{\widetilde{ad^{A}}(X)}](Y) &=&[\widetilde{\widetilde{%
ad^{A}}(X)}](\sum\limits_{j\in J:fini}\psi _{j}\cdot \delta
_{M^{A}}^{A}(\psi _{j}^{\prime })) \\
&=&\sum\limits_{j\in J:fini}\psi _{j}\cdot \lbrack \widetilde{\widetilde{%
ad^{A}}(X)}](\delta _{M^{A}}^{A}(\psi _{j}^{\prime })) \\
&=&\sum\limits_{j\in J:fini}\psi _{j}\cdot ([\widetilde{\widetilde{ad^{A}}(X)%
}]\circ \delta _{M^{A}}^{A})(\psi _{j}^{\prime })) \\
&=&\sum\limits_{j\in J:fini}\psi _{j}\cdot \lbrack \widetilde{ad^{A}}%
(X)](\psi _{j}^{\prime }) \\
&=&-\sum\limits_{j\in J:fini}\psi _{j}\cdot \omega _{M^{A}}^{A}(X,\delta
_{M^{A}}^{A}(\psi _{j}^{\prime })) \\
&=&-\omega _{M^{A}}^{A}(X,\sum\limits_{j\in J:fini}\psi _{j}\cdot \delta
_{M^{A}}^{A}(\psi _{j}^{\prime })) \\
&=&-\omega _{M^{A}}^{A}(X,Y)\text{;}
\end{eqnarray*}

3.%
\begin{eqnarray*}
\mathfrak{L}_{\widetilde{ad^{A}}[\delta _{M^{A}}^{A}(\varphi )]}\delta
_{M^{A}}^{A}(\psi ) &=&\mathfrak{L}_{ad^{A}(\varphi )}\delta
_{M^{A}}^{A}(\psi ) \\
&=&\delta _{M^{A}}^{A}[ad^{A}(\varphi )(\psi )] \\
&=&\delta _{M^{A}}^{A}(\{\varphi ,\psi \}_{M^{A}})\text{.}
\end{eqnarray*}
\end{proof}

\end{document}